\documentclass{birkjour}

\usepackage[utf8]{inputenc}
\usepackage[T1]{fontenc}
\usepackage{amsmath, amsfonts, amssymb, amsthm}
\usepackage[none]{hyphenat}
\usepackage{graphicx}
\usepackage{float}
\usepackage{times}
\usepackage{esint}
\usepackage{lipsum}
\usepackage{enumitem}
\usepackage{amsmath, amsthm, amscd, amsfonts, amssymb, graphicx, color}
\usepackage[bookmarksnumbered, colorlinks, plainpages]{hyperref}
\hypersetup{colorlinks=true,linkcolor=red, anchorcolor=green, citecolor=cyan, urlcolor=red, filecolor=magenta, pdftoolbar=true}

\usepackage[scr]{rsfso}

\newtheorem{theorem}{Theorem}[section]
\newtheorem{lemma}[theorem]{Lemma}
\newtheorem{proposition}[theorem]{Proposition}
\newtheorem{corollary}[theorem]{Corollary}
\theoremstyle{definition}
\newtheorem{definition}[theorem]{Definition}

\theoremstyle{remark}

\numberwithin{equation}{section}

\numberwithin{equation}{section}

\begin{document}

%
%
%
%
%
%
%
%
%

\title{On the Continuous embeddings between the fractional Haj{\l}asz-Orlicz-Sobolev spaces}

\author{Azeddine BAALAL}
\address{{Hassan 2 University, Department of Mathematics and Computer Science},
	{Ain Chock Faculty}, 
	{Casablanca},
	{Morocco}}
\email{abaalal@gmail.com}

\author{Mohamed BERGHOUT}
\address{{Laboratory of Partial Differential Equations, Algebra and Spectral Geometry, Higher School of Education and Training, Ibn Tofail University, P.O.Box 242-Kenitra 14000}, {Kenitra}, {Morocco}}
\email{Mohamed.berghout@uit.ac.ma; moh.berghout@gmail.com}

 \author{EL-Houcine OUALI}
 \address{{Hassan 2 University, Department of Mathematics and Computer Science},
 	{Ain Chock Faculty}, 
 	{Casablanca},
 	{Morocco}}
 \email{oualihoucine4@gmail.com}


\date{}

\begin{abstract}
	  Let $G$ be an Orlicz function and let $ \alpha, \beta, s$ be positive real numbers. Under certain conditions on the Orlicz function $ G $, we establish some continuous embeddings results between the fractional order Orlicz-Sobolev spaces defined on metric-measure spaces $W_s^{\alpha, G}(X, d, \mu)$ and the fractional Haj{\l}asz-Orlicz-Sobolev spaces $M^{\beta, G}(X,d,\mu)$.
\end{abstract}

\maketitle

\section{Introduction} 
A metric-measure space $ (X,d,\mu)$ is a metric space $ (X, d) $ with a Borel measure $ \mu $ such that $ 0 < \mu (B(x,r))  < \infty  $ for all $ x \in X $ and all $ r \in \left( 0,\infty\right)$. We will always assume
that metric spaces have at least two points. Sobolev spaces on metric measure spaces have been studied during the last two decades see e.g \cite{BaBj2011,Hp1996,HpKp2000}. The theory was generalized to Orlicz-Sobolev spaces on metric-measure spaces  in \cite{An2004, Ht2012, Mm2010, PlKaJoFs2012, Th2005}. These spaces play an important role in the so-called area of strongly nonlinear potential theory \cite{An2004}. Continuous and compact embeddings play a central role in the theory of Sobolev spaces. Those issues have been investigated by many researchers, we refer to \cite{AdHl1999, Hebey, di2012hitchhikers, zhou2015fractional, ArGpHp2020, Hp1996, Hp2003, karak2019measure, karak2020lower, GpSa2022}. 

The objective of this paper is to study the continuous embeddings of the fractional Haj{\l}asz-Orlicz-Sobolev spaces.

Let $(X,d,\mu)$ be a metric-measure space and $ G $ be an Orlicz function (cf. Section \ref{sec2} for precise definition). We define the Orlicz space $L^G(X, \mu)$ as follows
\begin{equation*}
L^G(X, \mu):=\left\lbrace u: X\rightarrow \mathbb{R} \textit{ measurable; } \Phi_{G}(u)< +\infty \right\rbrace, 
\end{equation*} 
where the modular $\Phi_{G}$ is defined as follows \begin{equation*}
\displaystyle \Phi_{G}(u):=\int_{X}G(|u(x)|)d \mu(x).
\end{equation*}
The Orlicz space $L^G(X, \mu)$  is a Banach space with the following norm, called the Luxemburg norm \begin{equation*}
\|u\|_{L^G(X, \mu)}:=\inf \left\lbrace \lambda >0 ; \Phi_{G}\left( \frac{u}{\lambda}\right) \leq 1 \right\rbrace.
\end{equation*}
Let $ \alpha,s>0$, we define the fractional order Orlicz-Sobolev space  $W_s^{\alpha, G}(X, d, \mu)$ as follows 
\begin{equation*}
W_s^{\alpha, G}(X, d, \mu):=\left\lbrace  u\in L^G(X, \mu); \Phi^{\alpha,G}_{s}(u)<+\infty \right\rbrace, 
\end{equation*}
where the modular $\Phi^{\alpha,G}_{s}$ is defined as 
\begin{equation*}
\displaystyle \Phi^{\alpha,G}_{s}(u):=\iint_{0<d(x, y)<1} G\left( \frac{|u(x)-u(y)|}{ d(x, y)^{\alpha}}\right) \frac{d \mu(y) d \mu(x)}{d(x, y)^{s}}.
\end{equation*}
The space  $W_s^{\alpha, G}(X, d, \mu)$ is a Banach space with the following norm 
\begin{equation}\label{1eq3}  
\|u\|_{W_s^{\alpha, G}(X, d, \mu)}:=\|u\|_{L^G(X, \mu)}+[u]_{W_s^{\alpha, G}(X, d, \mu)},
\end{equation}\label{1eq4}
where $[.]_{W_s^{\alpha, G}(X, d, \mu)}$ is the Gagliardo seminorm, defined by
\begin{equation}
[u]_{W_s^{\alpha, G}(X, d, \mu)}:=\inf\left\lbrace \lambda>0; \Phi^{\alpha,G}_{s}\left( \frac{u}{\lambda}\right) \leq 1 \right\rbrace.
\end{equation}
Furthermore, we say that $u\in \dot{W}_s^{\alpha, G}(X,d,\mu)$ if $u$ is finite almost everywhere and $[u]_{W_s^{\alpha, G}(X,d,\mu)}<\infty$. 

Let $\beta >0$, we say that a function $v\in L^{G}(X,\mu)$ belongs to the fractional Haj{\l}asz-Orlicz-Sobolev space $M^{\beta, G}(X,d,\mu)$ if there exists a non-negative function $h \in L^G(X,\mu)$, called a generalized $\beta$-gradient of $v$, such that
\begin{equation*}
\left| v(x)-v(y)\right| \leq d^{\beta}(x, y)(h(x)+h(y)) \quad \text{for} \ \mu\text{-almost all} \ x, y \in X.
\end{equation*}
We equip the space $M^{\beta, G}(X,d,\mu)$ with the norm
\begin{equation*}
\left\| v\right\| _{M^{\beta, G}(X,d,\mu)}:=\|v\|_{L^G(X, \mu)}+\inf \|h\|_{L^G(X, \mu)},
\end{equation*}
where the infimum is taken over all the generalized $\beta$-gradients of $v$. When $ G(t)=t^{p}$ with $ p>1$, those spaces were introduced by Haj{\l}asz for $\alpha=1$ (see \cite{Hp1996}) and by Yang  for $\alpha \neq 1$ (see \cite{Yd2003}). For more details on these spaces, we refer to \cite{ArGpHp2020, Hp2003, HpKp2000, Hp1996,ZxSq2022,MjOw2022,Yd2003}. 

The classical Sobolev embedding theorems for $ W^{s,p}(\mathbb{R}^{n})$ have different character when $ sp < n $, $ sp = n $, or $ sp > n $, see \cite{di2012hitchhikers}. Therefore, in the metric-measure context, in order to prove embedding theorems, we need a condition that would be the counterpart of the dimension of the space. It turns out that such a condition is provided by the lower bound and upper bound for the growth of the measure.\\ 
We say that the metric-measure space $(X, d, \mu)$ is lower Ahlfors $s$-regular if there exists some positive constant $b$ such that 
\begin{equation}\label{1eq1}
\mu(B(z, r)) \geq \frac{1}{b} r^s \text { for } r \in(0,1], z \in X .
\end{equation}
On the other hand, we say that $(X, d, \mu)$ is upper Ahlfors $s$-regular if there is some positive constant $b$ such that 
\begin{equation}\label{1eq2}
\mu(B(z, r)) \leq b r^s \text { for } r \in(0,1], z \in X.
\end{equation}
 For example, let $\theta>2$ and
\begin{equation*}
\Omega:=\left\{(x, y) \in \mathbb{R}^2: 0<y<1 \text { and }|x|<y^{\theta-1}\right\}.
\end{equation*}
Then the metric-measure space $(\Omega,d_{\mathbb{R}^{2}},\lambda_{2})$ is lower $\theta$-regular, see \cite{GpSa2022}. Here we regard $\Omega$ as a metric-measure space with the Euclidean metric $d_{\mathbb{R}^{2}}$, and the Lebesgue measure $\lambda_{2}$.

The first main result of this paper is the H\"{o}lder regularity for the Sobolev functions when $ \alpha p_{0}>s $, see Theorem \ref{3theo1} and Theorem \ref{3theo2}. The second main result of our paper is a continuous embedding between the spaces $W_s^{\alpha, G}(X, d, \mu)$ and $M^{\beta, G}(X,d,\mu)$, see Theorem \ref{3theo3}.  
		 
This paper is organized as follows. In Section \ref{sec2}, we give some useful properties for Orlicz functions. The proof of the main results is given in Section \ref{sec3}.
\section{Orlicz functions and some preliminaries results}\label{sec2}
In this section, we prove several useful properties of the Orlicz functions, which will be used throughout this paper. We start by recalling the definition of the well-known Orlicz functions.
\begin{definition}
A function $G: \mathbb{R}_{+} \rightarrow \mathbb{R}_{+}$ is called an Orlicz function if it has the following properties :
\begin{itemize}
	\item[$\left(H_1\right)$] $G$ is continuous, convex, increasing and $G(0)=0$. 
	\item[$\left(H_2\right)$] $G$ satisfies the $\Delta_2$ condition, that is, there exists $C>2$ such that $G(2 x) \leq C G(x)$ for all $x \in \mathbb{R}_{+}$. 
	\item[$\left(H_3\right)$] $G$ is super-linear at zero, that is $\displaystyle\lim _{x \rightarrow 0} \frac{G(x)}{x}=0$. 
\end{itemize}
\end{definition}
It is easy to check that Orlicz functions fulfill the following basic properties. See \cite{BjSa2019} for a proof of these facts. 
\begin{itemize}
	\item[$\left(P_1\right)$]
	(Regularity) $G$ is Lipschitz continuous.
	\item[$\left(P_2\right)$]  $G$ can be represented in the form    $\displaystyle G(x)=\int_0^x g(s) d s,$    
	
	where $g$ is a non-decreasing right continuous function.
	\item [$\left(P_3\right)$] (Subadditivity) Given $a, b \in \mathbb{R}_{+}$, 
	$
	G(a+b) \leq \frac{C}{2}(G(a)+G(b))
	$,  where $C >0$ is the constant in the $\Delta_2$ condition. 
	\item[$\left(P_4\right)$]  For any $0<b<1$ and $a>0$, it holds $G(a b) \leq b G(a)$. 
\end{itemize}
A typical example of the Orlicz function is $ G(t)=t^{p} $ with $ p>1$, we refer to \cite{krasnosel1961convex} for more examples of Orlicz functions. 

Throughout this paper, we assume that 
\begin{equation} \label{2eq1}
1< p_0:=\inf _{s > 0} \frac{s g(s)}{G(s)} \leq p^0:=\sup _{s > 0} \frac{s g(s)}{G(s)}<+\infty . 
\end{equation}
\begin{lemma}\label{2lem6}
	Let $ G $ be an Orlicz function and assume that $(\ref{2eq1})$ is satisfied. Then the following inequalities hold.	
	\begin{equation} \label{1GG1}
	G(\gamma t) \geq \gamma^{p_{0}} G(t) \textit{  for all } t>0 \textit{  and all } \gamma \geq 1,
	\end{equation}
	
	\begin{equation} \label{1GG2}
	G(\gamma t) \geq \gamma^{p^{0}} G(t) \textit{  for all } t>0 \textit{  and all } \gamma \in(0,1],
	\end{equation}
	
	\begin{equation} \label{1GG3}
	G(\gamma t) \leq \gamma^{p^{0}} G(t)  \textit{ for all } t>0 \textit{ and all } \gamma \geq 1,
	\end{equation}
	
	\begin{equation} \label{1GG4}
	G(t) \leq \gamma^{p_{0}} G\left(\frac{t}{\gamma}\right) \textit{  for all } t>0 \textit{  and all }\gamma \in(0,1].
	\end{equation}
\end{lemma}
For $\gamma \neq 1$, the proof of the lemma is given in \cite[Lemma~2.2]{ABSS2020}. For the convenience of the reader we present it here.
\begin{proof} If $\gamma=1$, the previous inequalities are trivial. Suppose that $\gamma \neq 1$, 
	since $p_{0} \leq \frac{t g(t)}{G(t)}$ for all $t>0$, it follows that for letting $\gamma >1$, we have
	$$
	\log \left(G(\gamma t)\right)-\log \left(G(t)\right)=\int_t^{\gamma t} \frac{g(\tau)}{G(\tau)} d \tau \geq \int_t^{\gamma t} \frac{p_{0}}{\tau} d \tau=\log \left(\gamma^{p_{0}}\right) .
	$$
	Thus, we deduce
	\begin{equation}\label{1G1}
	G(\gamma t) \geq \gamma^{p_{0}} G(t) \text { for all } t>0 \text { and } \gamma>1 .
	\end{equation}
	
	Now, we show that $G(\gamma t) \leq \gamma^{p^{0}} G(t) \quad$ for all $t>0$ and $\gamma >1$. 
	Indeed, since $p^{0} \geq \frac{t g(t)}{G(t)}$ for all $t>0$, it follows that for all $\gamma>1$, we have
	$$
	\log (G(\gamma t))-\log (G(t))=\int_t^{\gamma t} \frac{g(\tau)}{G(\tau)} d \tau \leq \int_t^{\gamma t} \frac{p^{0}}{\tau} d \tau=\log \left(\gamma^{p^{0}}\right).
	$$
	Thus, we deduce
	\begin{equation}\label{1G2}
	G(\gamma t) \leq \gamma^{p^{0}} G(t) \text { for all } t>0 \text { and } \gamma>1.
	\end{equation}
	
	By the same argument in the proof of \eqref{1G1} and \eqref{1G2}, we have 
	\begin{equation}
	G(\gamma t) \geq \gamma^{p^{0}} G(t) \textit{  for all  } t>0 \textit{  and } \gamma \in(0,1),
	\end{equation}
	and 
	\begin{equation}
	G(t) \leq \gamma^{p_{0}} G\left(\frac{t}{\gamma}\right) \textit{  for all  } t>0 \textit{ and } \gamma \in(0,1).
	\end{equation}
	This concludes the proof.
\end{proof}
\begin{proposition}\label{2pro1}
	Let  $s, \alpha>0$, $G$ be an Orlicz function and suppose that $(X, d, \mu)$ is an upper Ahlfors $\theta$-regular space, where $\alpha p^{0} <\theta-s$. Then, $W_s^{\alpha, G}(X,d,\mu)=L^G(X, \mu)$ with equivalent norms.
\end{proposition}
\begin{proof}
	Let $u \in L^G(X, \mu)$ and $\lambda>0$, then by using the Fubini's Theorem and the inequality \eqref{1GG3}  we have 
	$$
	\begin{aligned}
	\Phi^{\alpha,G}_{s}\left(\frac{u}{\lambda}\right)& =\iint_{0<d(x, y)<1} G\left( \frac{|u(x)-u(y)|}{\lambda d(x, y)^{\alpha}}\right) \frac{d \mu(y) d \mu(x)}{d(x, y)^{s}} \\ & \leq \iint_{0<d(x, y)<1} G\left( \frac{2|u(x)|}{2\lambda d(x, y)^{\alpha}}+\frac{2|u(y)|}{2\lambda d(x, y)^{\alpha}}\right) \frac{d \mu(y) d \mu(x)}{d(x, y)^{s}}\\
	& \leq 2^{p^{0}}\int_X  \int_{B(x, 1) \backslash\{x\}}G\left(  \frac{|u(x)|}{\lambda d(x, y)^{\alpha}} \right) \frac{d \mu(y) d \mu(x)}{d(x, y)^{s}}  \\
	&+  2^{p^{0}} \int_X  \int_{B(y, 1) \backslash\{y\}}G\left(  \frac{|u(y)|}{\lambda d(x, y)^{\alpha}} \right) \frac{d \mu(x) d \mu(y)}{d(x, y)^{s}}  \\
	&\leq 2^{p^{0}+1} \int_XG\left( \frac{|u(x)|}{\lambda}\right) \int_{B(x, 1) \backslash\{x\}}\frac{1}{d(x, y)^{s+\alpha p^{0}}} d \mu(y)d\mu(x) \\
	& \leq  2^{s+\alpha p^{0} +p^{0}+1} \int_XG\left(\frac{|u(x)|}{\lambda}\right) \left(\sum_{k=0}^{\infty} \mu\left(B\left(x, 2^{-k}\right)\right) 2^{k(s+\alpha p^{0})}\right) d \mu(x) \\
	& \leq b 2^{s+\alpha p^{0}+p^{0}+1}\int_{X} G\left(\frac{|u(x)|}{\lambda}\right) d \mu(x) \sum_{k=0}^{\infty} 2^{k(\alpha p^{0}+s-\theta)}\\
	&=\frac{b 2^{s+\alpha p^{0}+p^{0}+1}}{1-2^{\alpha p^{0}+s-\theta}}\int_{X} G\left(\frac{|u(x)|}{\lambda}\right) d \mu(x),
	\end{aligned}
	$$
	and this implies that $$ [u]_{W^{\alpha,G}_{s}(X,d,\mu)}\leq \beta \|u\|_{L^{G}(X,d,\mu)}, \textit{ with }  \beta=\frac{b 2^{s+\alpha p^{0}+p^{0}+1}}{1-2^{\alpha p^{0}+s-\theta}},$$
	where we have used the decomposition $B(x, 1) \backslash\{x\}=$ $\bigcup_{k=0}^{\infty} B\left(x, 2^{-k}\right) \backslash B\left(x, 2^{-(k+1)}\right).$\newline
This completes the proof.
\end{proof}
In the sequel, let $(X, d, \mu)$ be a lower Ahlfors s-regular space and let $B$ be the ball with the center $z\in X$ and the radius $R>0$.\\
 The median value of a measurable function $u$ finite a.e. on $B$ is defined by 
$$
m_u(B)=\max \left\lbrace t \in \mathbb{R}: \mu(\{x \in B: u(x)<t\}) \leq \frac{\mu(B)}{2} \right\rbrace  .
$$
From \cite{GpSa2022}, we have the following result.
\begin{proposition}\label{2pro2}
	Let $u$ be a measurable function finite a.e. on $B$, then
	\begin{itemize}
		
		\item[$(a)$] if $c \in \mathbb{R}$, then $m_u(B)-c=m_{u-c}(B)$,
		\item[$(b)$] if $c>0$, then $c m_u(B)=m_{c u}(B)$.
	\end{itemize}
\end{proposition}
\begin{lemma}
Let $G$ be an Orlicz function and let $u$ be a measurable function finite a.e. Then, for all $ c \in \mathbb{R}$, we have
	\begin{equation}\label{2eq9}
	\left|m_u(B)-c\right| \leq G^{-1}\left( 2 \fint_B G\left(|u(x)-c|\right) d \mu(x)\right) .
	\end{equation}
\end{lemma}
\begin{proof}
Let us first prove the following inequality
	\begin{equation}\label{2eq10}
	\left|m_u(B)\right| \leq m_{|u|}(B) .
	\end{equation}
	If $m_u(B)>0$, then we have
	$$
	\left\{x \in B:|u(x)|<m_u(B)\right\} \subseteq\left\{x \in B: u(x)<m_u(B)\right\},
	$$
	thus the inequality \eqref{2eq10} follows.\\
	 If $m_u(B) \leq 0$, then we have 
	\begin{equation}\label{2equ10d1}
	\left\{x \in B:|u(x)|<\left|m_u(B)\right|\right\} \subseteq\left\{x \in B: u(x)>m_u(B)\right\},
	\end{equation}
	and by the definition of the median value we have
	\begin{equation}\label{2eq10d}
	\mu\left(\left\{x \in B: u(x) \leq m_u(B)\right\}\right) \geq \frac{\mu(B)}{2}.
	\end{equation}
	Hence from the inclusion \eqref{2equ10d1} and the inequality \eqref{2eq10d} we get \eqref{2eq10}.\newline
 Now, let us show the following inequality
	\begin{equation}\label{2eq11}
	m_{|u-c|}(B) \leq G^{-1}\left(2 \fint_B G\left(|u(x)-c|\right)  d \mu(x)\right) . 
	\end{equation}
	Let $\eta>2$ and $\delta=\fint_{B}G\left(|u(x)-c|\right) d \mu(x)$. Then by general Chebyshev's inequality, we have
	$$
	\begin{aligned}
	\mu\left( \left\lbrace x \in B:|u(x)-c| \geq G^{-1}\left( \eta \delta \right)  \right\rbrace \right)  &\leq \frac{1}{\eta \delta} \int_B G\left( |u(x)-c|\right)  d \mu(x)\\
	&=\frac{1}{\eta} \mu(B)<\frac{1}{2} \mu(B) .
	\end{aligned}
	$$
	Thus,
	$$
	m_{|u-c|}(B) \leq G^{-1}\left(\eta \fint_B G\left( |u(x)-c|\right)  d \mu(x)\right) .
	$$
	Then by passing to the limit $\eta \rightarrow 2$ we get \eqref{2eq11}.\newline Hence, by Proposition \ref{2pro2} and the inequalities \eqref{2eq10} and \eqref{2eq11} we obtain the result.
\end{proof}
\begin{corollary}
Let $G$ be an Orlicz function and let $u$ be a measurable function finite a.e. Then for all $t$ such that $|u(t)|<\infty$ we get
	\begin{equation}\label{2eq12}
	G\left( \left|m_u(B(z, r))-u(t)\right|\right)  \leq 2 b r^{-s} \int_{B(z, r)} G\left(|u(t)-u(y)|\right)  d \mu(y),
	\end{equation}
	where $z \in X$ and $r \leq 1$.
\end{corollary}
Now let us prove the following lemma.
\begin{lemma}\label{2lem8}
Let $G$ be an Orlicz function and let $u \in \dot{W}_s^{\alpha, G}(X,d,\mu)$ for some $\alpha >0$. Then, there exists a null set $H$ such that $\lim _{r \rightarrow 0} m_u(B(x, r))=u(x)$ for all $x \in X \backslash H$.
\end{lemma}
\begin{proof}
	Let us consider the following set
	$$
	H=\left\{x \in X: \int_{B(x, 1)} G\left( \frac{|u(x)-u(y)|}{d(x, y)^{\alpha}}\right) \frac{d \mu(y)}{d(x, y)^{s}}=+\infty \text { or }|u(x)|=+\infty\right\}.
	$$
 By the Fubini's Theorem, we have $\displaystyle \int_{B(x, 1)} G\left(\frac{|u(x)-u(y)|}{d(x, y)^{\alpha}}\right)  \frac{d\mu(y)}{d(x, y)^{s}} \in L^1(X, \mu)$, then $\mu(H)=0$. Hence, for all $x \in X \backslash H$ and $r \leq 1$, by \eqref{2eq12} and $\eqref{1GG4}$ we have
\begin{equation*}
	\begin{aligned}
	G\left( \left|u(x)-m_u(B(x, r))\right|\right)  & \leq 2 b r^{-s} \int_{B(x, r)}G\left(|u(x)-u(y)|\right)  d \mu(y) \\
	& \leq 2 b r^{-s} \int_{B(x, r)} G\left( \frac{|u(x)-u(y)|}{d(x, y)^{\alpha}}\right) d(x, y)^{\alpha p_{0}+ s} \frac{d \mu(y)}{d(x, y)^{s}} \\
	& \leq 2 b r^{\alpha p_{0} } \int_{B(x, 1)} G\left(\frac{|u(x)-u(y)|}{d(x, y)^{\alpha}}\right) \frac{d \mu(y)}{d(x, y)^{s}}.
	\end{aligned} 
\end{equation*}
 Since $\displaystyle 2 b r^{\alpha p_{0} } \int_{B(x, 1)} G\left(\frac{|u(x)-u(y)|}{d(x, y)^{\alpha}}\right) \frac{d \mu(y)}{d(x, y)^{s}} $ tends to $0$  as $r$ tends to $0$, the proof follows.
\end{proof}
\begin{proposition}\label{prop2.8} Let $G$ be an Orlicz function and let $\alpha >0$. Then for all $z \in X$, $0<r_1 \leq r_2<1-r_1$ and $u \in \dot{W}_s^{\alpha, G}(X,d,\mu)$ we have
	\begin{equation}\label{2eq13}
	\begin{aligned}
	&G\left(\left|m_u\left(B\left(z, r_1\right)\right)-m_u\left(B\left(z, r_2\right)\right)\right|\right)\\&  \leq 4 b^2 2^{s+\alpha p_{0}}[u]_{W^{\alpha,G}_{s}(X,d,\mu)}^{p_{0}}\left(r_1^{\alpha p_{0}-s}+\frac{r_2^{\alpha p_{0}}}{r_1^s}\right)
	\end{aligned}
	\end{equation}
	if $[u]_{W^{\alpha,G}_{s}(X,d,\mu)}\leq 1$, and  
	\begin{equation}\label{2eq14}
	\begin{aligned}
	&G\left(\left|m_u\left(B\left(z, r_1\right)\right)-m_u\left(B\left(z, r_2\right)\right)\right|\right)\\&  \leq 4 b^2 2^{s+\alpha p_{0}}[u]_{W^{\alpha,G}_{s}(X,d,\mu)}^{p^{0}}\left(r_1^{\alpha p_{0}-s}+\frac{r_2^{\alpha p_{0}}}{r_1^s}\right)
	\end{aligned}
	\end{equation}
	if $[u]_{W^{\alpha,G}_{s}(X,d,\mu)}\geq 1$.
\end{proposition}
\begin{proof} 
	By using the inequalities \eqref{2eq9}, \eqref{2eq12} and \eqref{1GG4}, we have
   $$
    \begin{aligned}
&G\left( \left|m_u\left(B\left(z, r_1\right)\right)-m_u\left(B\left(z, r_2\right)\right)\right|\right)\\ &\leq 2 b r_1^{-s} \int_{B\left(z, r_1\right)}G\left( \left|u(x)-m_u\left(B\left(z, r_2\right)\right)\right|\right)  d \mu(x)\\
	&\leq 4 b^2 r_1^{-s} r_2^{-s} \int_{B\left(z, r_1\right)} \int_{B\left(z, r_2\right)}G\left(|u(x)-u(y)|\right)  d \mu(y) d \mu(x) \\
	& \leq 4 b^2 r_1^{-s} r_2^{-s} \int_{B\left(z, r_1\right)} \int_{B\left(z, r_2\right)}G\left(\frac{|u(x)-u(y)|}{d(x, y)^{\alpha}} \right)d(x, y)^{\alpha p_{0}}d(x, y)^{s} \frac{d \mu(y) d \mu(x)}{d(x, y)^{s}}\\
	& \leq 4 b^2 r_1^{-s} r_2^{-s} \int_{B\left(z, r_1\right)} \int_{B\left(z, r_2\right)} G\left( \frac{|u(x)-u(y)|}{d(x, y)^{\alpha}}\right) \left(r_1+r_2\right)^{s+\alpha p_{0}} \frac{d \mu(y) d \mu(x)}{d(x, y)^{s}}
	\end{aligned}
	$$
	\begin{equation}\label{equ2.15}
	\begin{aligned}
	& \leq 4 b^2 r_1^{-s} r_2^{-s}\left(r_1+r_2\right)^{s+\alpha p_{0}}\iint_{0<d(x, y)<1} G\left( \frac{|u(x)-u(y)|}{ d(x, y)^{\alpha}}\right) \frac{d \mu(y) d \mu(x)}{d(x, y)^{s}} \\
	& =4 b^2 \frac{\left(r_1+r_2\right)^s}{r_2^s} \frac{\left(r_1+r_2\right)^{\alpha p_{0}}}{r_1^s}\iint_{0<d(x, y)<1} G\left(\frac{ |u(x)-u(y)|}{d(x, y)^{\alpha}}\right) \frac{d \mu(y) d \mu(x)}{d(x, y)^{s}} \\
	& \leq 4 b^2 2^{s+\alpha p_{0}}\\ &\times\left(r_1^{\alpha p_{0} -s}+\frac{r_2^{\alpha p_{0} }}{r_1^s}\right)\iint_{0<d(x, y)<1} G\left( \frac{|u(x)-u(y)|}{d(x, y)^{\alpha}}\right) \frac{d \mu(y) d \mu(x)}{d(x, y)^{s}}.
	\end{aligned}
	\end{equation}
If $[u]_{W^{\alpha,G}_{s}(X,d,\mu)}=0$, the inequality  \eqref{2eq13} is trivial.\newline
We suppose that $[u]_{W^{\alpha,G}_{s}(X,d,\mu)}\in (0,1]$, then by \eqref{equ2.15} and \eqref{1GG4}  we have
$$
\begin{aligned}
	&G\left( \left|m_u\left(B\left(z, r_1\right)\right)-m_u\left(B\left(z, r_2\right)\right)\right|\right)\\ &\leq  4 b^2 2^{s+\alpha p_{0}}\left(r_1^{\alpha p_{0} -s}+\frac{r_2^{\alpha p_{0} }}{r_1^s}\right)\\& \times [u]_{W^{\alpha,G}_{s}(X,d,\mu)}^{p_{0}}\iint_{0<d(x, y)<1} G\left( \frac{|u(x)-u(y)|}{[u]_{W^{\alpha,G}_{s}(X,d,\mu)}d(x, y)^{\alpha}}\right) \frac{d \mu(y) d \mu(x)}{d(x, y)^{s}}\\
	&\leq 4 b^2 2^{s+\alpha p_{0}}[u]_{W^{\alpha,G}_{s}(X,d,\mu)}^{p_{0}}\left(r_1^{\alpha p_{0} -s}+\frac{r_2^{\alpha p_{0} }}{r_1^s}\right).
\end{aligned}
$$
Now we suppose that $[u]_{W^{\alpha,G}_{s}(X,d,\mu)}\geq 1$, then by \eqref{equ2.15} and \eqref{1GG3} we have 
	$$ 
	\begin{aligned}
	&G\left( \left|m_u\left(B\left(z, r_1\right)\right)-m_u\left(B\left(z, r_2\right)\right)\right|\right)\\
	&\leq  4 b^2 2^{s+\alpha p_{0}}\left(r_1^{\alpha p_{0} -s}+\frac{r_2^{\alpha p_{0} }}{r_1^s}\right) \\& \times\iint_{0<d(x, y)<1} G\left( [u]_{W^{\alpha,G}_{s}(X,d,\mu)} \frac{|u(x)-u(y)|}{[u]_{W^{\alpha,G}_{s}(X,d,\mu)}d(x, y)^{\alpha}}\right) \frac{d \mu(y) d \mu(x)}{d(x, y)^{s}}\\ 
	&\leq  4 b^2 2^{s+\alpha p_{0}}\left(r_1^{\alpha p_{0} -s}+\frac{r_2^{\alpha p_{0} }}{r_1^s}\right)\\&\times [u]_{W^{\alpha,G}_{s}(X,d,\mu)}^{p^{0}}\iint_{0<d(x, y)<1} G\left( \frac{|u(x)-u(y)|}{[u]_{W^{\alpha,G}_{s}(X,d,\mu)}d(x, y)^{\alpha}}\right) \frac{d \mu(y) d \mu(x)}{d(x, y)^{s}}\\
	&\leq  4 b^2 2^{s+\alpha p_{0}}[u]_{W^{\alpha,G}_{s}(X,d,\mu)}^{p^{0}}\left(r_1^{\alpha p_{0} -s}+\frac{r_2^{\alpha p_{0} }}{r_1^s}\right).
	\end{aligned}
	$$
This concludes the proof.
\end{proof}
\begin{proposition}
Let $G$ be an Orlicz function such that $G^{-1}(xy)\leq x^{\frac{1}{p_{0}}}G^{-1}(y)$ if $x\leq 1$, and $G^{-1}(xy)\leq x^{\frac{1}{p^{0}}}G^{-1}(y)$ if $x\geq 1$ and let $\alpha >0$. Then for all $z \in X$, $R\in \left(0,\frac{2}{3}\right)$,  $i, j \in \mathbb{N}$, $0 \leq i<j$ and $u \in \dot{W}_s^{\alpha, G}(X,d,\mu)$ we have
	\begin{equation}\label{2eq15} 
	\begin{aligned}
	&\left|m_u\left(B\left(z, \frac{R}{2^j}\right)\right)-m_u\left(B\left(z, \frac{R}{2^i}\right)\right)\right|\\& \leq A [u]_{W^{\alpha, G}_{s}(X,d,\mu)}\sum_{l=i}^{j-1} G^{-1}\left( \frac{R^{\alpha p_{0}-s}}{2^{l(\alpha p_{0}-s)}}\right),
	\end{aligned}
	\end{equation}
	where $A=\displaystyle \max\left\lbrace \left(4^{s+1} b^2\right)^{\frac{1}{ p_{0}}}\left(1+2^{\alpha p_{0}}\right)^{\frac{1}{ p_{0}}},\left(4^{s+1} b^2\right)^{\frac{1}{p^{0}}}\left(1+2^{\alpha p_{0}}\right)^{\frac{1}{p^{0}}}\right\rbrace $.
\end{proposition}
\begin{proof}
	Let us take $r_1=\displaystyle \frac{R}{2^{l+1}} $ and $r_2=\displaystyle \frac{R}{2^l}$, where $l\in \mathbb{N}$. Then, by the proposition \ref{prop2.8}, we get
	$$
	\begin{aligned}
	&G\left( \left|m_u\left(B\left(z, \frac{R}{2^{l+1}}\right)\right)-m_u\left(B\left(z, \frac{R}{2^l}\right)\right)\right|\right) \\ & \leq 4 b^2 2^{s+\alpha p_{0}}[u]_{W^{\alpha, G}_{s}(X,d,\mu)}^{p_{0}}\left(\frac{1}{2^{\alpha p_{0}-s}} \frac{R^{\alpha p_{0}-s}}{2^{l(\alpha p_{0}-s)}}+2^s \frac{R^{\alpha p_{0}-s}}{2^{l(\alpha p_{0}-s)}}\right) \\
	& =4^{s+1} b^2 (1+2^{\alpha p_{0}})[u]_{W^{\alpha, G}_{s}(X,d,\mu)}^{p_{0}} \frac{R^{\alpha p_{0}-s}}{2^{l(\alpha p_{0}-s)}} 
	\end{aligned}
	$$
	if $[u]_{W^{\alpha, G}_{s}(X,d,\mu)} \leq 1$, and 
	$$
	\begin{aligned}
	&G\left( \left|m_u\left(B\left(z, \frac{R}{2^{l+1}}\right)\right)-m_u\left(B\left(z, \frac{R}{2^l}\right)\right)\right|\right)\\  & \leq  4 b^2 2^{s+\alpha p_{0}}[u]_{W^{\alpha, G}_{s}(X,d,\mu)}^{p^{0}}\left(\frac{1}{2^{\alpha p_{0}-s}} \frac{R^{\alpha p_{0}-s}}{2^{l(\alpha p_{0}-s)}}+2^s \frac{R^{\alpha p_{0}-s}}{2^{l(\alpha p_{0}-s)}}\right) \\
	& =4^{s+1} b^2 (1+2^{\alpha p_{0}})[u]_{W^{\alpha, G}_{s}(X,d,\mu)}^{p^{0}} \frac{R^{\alpha p_{0}-s}}{2^{l(\alpha p_{0}-s)}} 
	\end{aligned}
	$$
	if $[u]_{W^{\alpha, G}_{s}(X,d,\mu)} \geq 1$.\newline
Then in both cases, we have $$ \left|m_u\left(B\left(z, \frac{R}{2^{l+1}}\right)\right)-m_u\left(B\left(z, \frac{R}{2^l}\right)\right)\right|\leq   A [u]_{W^{\alpha, G}_{s}(X,d,\mu)}G^{-1}                   \left(  \frac{R^{\alpha p_{0}-s}}{2^{l(\alpha p_{0}-s)}} \right) $$
	where $A=\max\left\lbrace \left(4^{s+1} b^2\right)^{1 / p_{0}}\left(1+2^{\alpha p_{0}}\right)^{1 / p_{0}},\left(4^{s+1} b^2\right)^{1 / p^{0}}\left(1+2^{\alpha p_{0}}\right)^{1 / p^{0}}\right\rbrace $.\newline 
Hence, by the triangle inequality we have
$$
\begin{aligned}
	&\left|m_u\left(B\left(z, \frac{R}{2^j}\right)\right)-m_u\left(B\left(z, \frac{R}{2^i}\right)\right)\right|\\ & \leq \sum_{l=i}^{j-1}\left|m_u\left(B\left(z, \frac{R}{2^{l+1}}\right)\right)-m_u\left(B\left(z, \frac{R}{2^l}\right)\right)\right| \\
	& \leq A [u]_{W^{\alpha, G}_{s}(X,d,\mu)} \sum_{l=i}^{j-1} G^{-1}\left( \frac{R^{\alpha p_{0}-s}}{2^{l(\alpha p_{0}-s)}}\right),
\end{aligned}
$$ 
which completes the proof.
\end{proof}
\section{Continuous embeddings}\label{sec3}
This section is aimed at proving some results concerning  continuous embeddings between the fractional order Orlicz-Sobolev space defined on metric-measure spaces $W_s^{\alpha, G}(X, d, \mu)$, and the fractional Haj{\l}asz-Orlicz-Sobolev space $M^{\beta, G}(X,d,\mu)$ where $G$ is an Orlicz function.
\begin{theorem}\label{3theo1}
Let $G$ be an Orlicz function such that $G^{-1}(xy)\leq x^{\frac{1}{p_{0}}}G^{-1}(y)$ if $x\leq 1$ and  $G^{-1}(xy)\leq x^{\frac{1}{p^{0}}}G^{-1}(y)$ if $x\geq 1$, and let $s, \alpha>0$ such that $\alpha p_{0}>s$. Then, there exists a constant $C >0$ such that any $u \in \dot{W}_s^{\alpha, G}(X,d,\mu)$ equals a.e. to a continuous function $\widetilde{u}$ and we have the following inequality
	\begin{equation}\label{3eq1}
	\displaystyle\sup _{ 0<d(x, y)<\frac{1}{3}} \frac{\left|\widetilde{u}(x)-\widetilde{u}(y)\right|}{d(x, y)^{\alpha-\frac{s}{p_{0}}}} \leq C [u]_{W_s^{\alpha, G}(X,d,\mu)} .
	\end{equation}
\end{theorem}
\begin{proof}
	Let $z \in X$ and $0\leq R < \frac{1}{3}$. Since $\alpha p_{0}>s$, the inequality \eqref{2eq15} implies that  for all $i,j\in \mathbb{N}$ such that $0 \leq i<j$, we have
	\begin{equation}\label{3eq2}
	\left|m_u\left(B\left(z, \frac{R}{2^i}\right)\right)-m_u\left(B\left(z, \frac{R}{2^j}\right)\right)\right| \leq \widetilde{A} [u]_{W^{\alpha, G}_{s}(X,d,\mu)}\frac{R^{\alpha -\frac{s}{p_{0}}}}{2^{i\left(\alpha -\frac{s}{p_{0}} \right)}},
	\end{equation}
	where $\tilde{A}=G^{-1}(1)A /\displaystyle\left(1-2^{\frac{s}{p_{0}}-\alpha}\right)>0$. Therefore, we get that $\left\{m_u\left(B\left(z, \frac{R}{2^k}\right)\right)\right\}_{k=0}^{\infty}$ is a Cauchy sequence, thus is convergent. By Lemma \ref{2lem8} we have that $m_u\left(B\left(x, \frac{R}{2^k}\right)\right)$ converges to $u(x)$ for every $x \in X \backslash H$, where $\mu(H)=0$.\newline
	Let us take $i=0$ in \eqref{3eq2} and pass to the limit $j \rightarrow \infty$, then we get for all $z \in X \backslash H$
	\begin{equation}\label{3eq3}
	\left|m_u(B(z, R))-u(z)\right| \leq \widetilde{A} [u]_{W^{\alpha, G}_{s}(X,d,\mu)}R^{\alpha -\frac{s}{p_{0}}} .
	\end{equation}
	Let $z, w \in X \backslash H$ be such that $0<d(z, w)<\frac{1}{3}$ and let $R=d(z, w)$. Then, by \eqref{3eq3} we have 
	\begin{equation}\label{3eq4}
	\begin{aligned}
	|u(z)-u(w)| & \leq\left|u(z)-m_u(B(z, R))\right|+\left|u(w)-m_u(B(w, R))\right|\\
	&+\left|m_u(B(z, R))-m_u(B(w, R))\right| \\
	& \leq 2 \widetilde{A} [u]_{W^{\alpha, G}_{s}(X,d,\mu)} d(z, w)^{\alpha-\frac{s}{p_{0}}}\\&+\left|m_u(B(z, R))-m_u(B(w, R))\right| .
	\end{aligned}
	\end{equation}
	
 By the inequalities \eqref{2eq9} and \eqref{2eq12} we have
	\begin{equation}\label{3eqq44}
	\begin{aligned}
	&\left|m_u(B(z, R))-m_u(B(w, R))\right| \\
	& \leq G^{-1}\left(2 \fint_{B(z, R)} G\left|u(x)-m_u(B(w, R))\right|) d \mu(x)\right) \\
	& \leq G^{-1}\left(2 \fint_{B(z, R)} 2 b R^{-s} \int_{B(w, R)}G(|u(x)-u(y)|) d \mu(y) d \mu(x)\right) \\
	& \leq G^{-1}\left(4 b^2 R^{-2 s} \int_{B(z, R)} \int_{B(w, R)}G(|u(x)-u(y)|) d \mu(y) d \mu(x)\right).
	\end{aligned}
	\end{equation}
	Then by using the inequalities \eqref{3eqq44} and \eqref{1GG4} with $\gamma=d(x, y)^{\alpha}\in (0,1]$ and the estimate $d(x, y) \leq d(x, z)+d(z, w)+d(w, y) \leq 3 R$, we obtain 
	\begin{equation}\label{eqq3.54}
	\begin{aligned}
	&\left|m_u(B(z, R))-m_u(B(w, R))\right|\\
	 & \leq G^{-1}\left(4 b^2 R^{-2 s} \int_{B(z, R)} \int_{B(w, R)}G\left( \frac{|u(x)-u(y)|}{d(x, y)^{\alpha}}\right)d(x, y)^{\alpha p_{0}+s}\frac{d \mu(y) d \mu(x)}{d(x, y)^{s}}  \right)\\
	& \leq G^{-1}\left(4 b^2 3^{s+\alpha p_{0}} R^{\alpha p_{0}-s} \int_{B(z, R)} \int_{B(w, R)} G\left( \frac{|u(x)-u(y)|}{d(x, y)^{\alpha}}\right)  \frac{d \mu(y) d \mu(x)}{d(x, y)^{s}} \right)\\
	& \leq B d(z, w)^{\alpha-\frac{s}{p_{0}}}G^{-1}\left(\iint_{d(x, y)<1} G\left( \frac{|u(x)-u(y)|}{d(x, y)^{\alpha}}\right)  \frac{d \mu(y) d \mu(x)}{d(x, y)^{s}}\right),
	\end{aligned}
	\end{equation}
	where $B=\max\left\lbrace \left( 4 b^2 3^{s+\alpha p_{0}}\right) ^{\frac{1}{p_{0}}},\left( 4 b^2 3^{s+\alpha p_{0}}\right) ^{\frac{1}{p^{0}}}\right\rbrace $.\newline
	
	If $[u]_{W^{\alpha, G}_{s}(X,d,\mu)}=0$,  the inequality \eqref{3eq1} is trivial.\newline
	
We suppose that $[u]_{W^{\alpha, G}_{s}(X,d,\mu)}\in(0,1]$, then from the inequalities \eqref{eqq3.54} and \eqref{1GG4} we get that
	$$
	\begin{aligned}
	&\left|m_u(B(z, R))-m_u(B(w, R))\right|\\ & \leq B d(z, w)^{\alpha-\frac{s}{p_{0}}}\\
	&\times G^{-1}\left(\iint_{d(x, y)<1} G\left( \frac{|u(x)-u(y)|}{[u]_{W^{\alpha, G}_{s}(X,d,\mu)}d(x, y)^{\alpha}}\right)[u]_{W^{\alpha, G}_{s}(X,d,\mu)}^{p_{0}}  \frac{d \mu(y) d \mu(x)}{d(x, y)^{s}}\right)\\
	& \leq B d(z, w)^{\alpha-\frac{s}{p_{0}}}[u]_{W^{\alpha, G}_{s}(X,d,\mu)}\\
	&\times G^{-1}\left(\iint_{d(x, y)<1} G\left( \frac{|u(x)-u(y)|}{[u]_{W^{\alpha, G}_{s}(X,d,\mu)}d(x, y)^{\alpha}}\right)  \frac{d \mu(y) d \mu(x)}{d(x, y)^{s}}\right)\\
	& \leq G^{-1}(1) B d(z, w)^{\alpha-\frac{s}{p_{0}}}[u]_{W^{\alpha, G}_{s}(X,d,\mu)}.
	\end{aligned}
	$$
	
	Now we suppose that $[u]_{W^{\alpha, G}_{s}(X,d,\mu)}\geq 1$, then from the inequalities \eqref{eqq3.54} and \eqref{1GG3} we get that
	
	$$
	\begin{aligned}
	&\left|m_u(B(z, R))-m_u(B(w, R))\right|\\ & \leq B d(z, w)^{\alpha-\frac{s}{p_{0}}}\\
	&\times G^{-1}\left(\iint_{d(x, y)<1} G\left([u]_{W^{\alpha, G}_{s}(X,d,\mu)} \frac{|u(x)-u(y)|}{[u]_{W^{\alpha, G}_{s}(X,d,\mu)}d(x, y)^{\alpha}}\right)  \frac{d \mu(y) d \mu(x)}{d(x, y)^{s}}\right)
	\\ & \leq B d(z, w)^{\alpha-\frac{s}{p_{0}}} \\
	&\times G^{-1}\left(\iint_{d(x, y)<1} G\left( \frac{|u(x)-u(y)|}{[u]_{W^{\alpha, G}_{s}(X,d,\mu)}d(x, y)^{\alpha}}\right)[u]_{W^{\alpha, G}_{s}(X,d,\mu)}^{p^{0}}  \frac{d \mu(y) d \mu(x)}{d(x, y)^{s}}\right)
	\end{aligned}
	$$
	
	$
	\begin{aligned}
	& \leq B d(z, w)^{\alpha-\frac{s}{p_{0}}}[u]_{W^{\alpha, G}_{s}(X,d,\mu)}\\
	&\times G^{-1}\left(\iint_{d(x, y)<1} G\left( \frac{|u(x)-u(y)|}{[u]_{W^{\alpha, G}_{s}(X,d,\mu)}d(x, y)^{\alpha}}\right)  \frac{d \mu(y) d \mu(x)}{d(x, y)^{s}}\right)\\
	& \leq G^{-1}(1) B d(z, w)^{\alpha-\frac{s}{p_{0}}}[u]_{W^{\alpha, G}_{s}(X,d,\mu)}.
	\end{aligned}
	$\newline
	Then in both cases we have $$\left|m_u(B(z, R))-m_u(B(w, R))\right| \leq G^{-1}(1) B d(z, w)^{\alpha-\frac{s}{p_{0}}}[u]_{W^{\alpha, G}_{s}(X,d,\mu)}.$$
	Then, by \eqref{3eq4} we get
	$$
	|u(z)-u(w)| \leq C [u]_{W^{\alpha, G}_{s}(X,d,\mu)} d(z, w)^{\alpha-\frac{s}{p_{0}}}
	$$
	for all $z, w \in X \backslash H$ such that $d(z, w)<\frac{1}{3}$, where $C =2 \widetilde{A}+G^{-1}(1)B$.\newline
	Now, we extend $u$ to a function $\widetilde{u}$ which is continuous on $X$ and satisfies \eqref{3eq1}.\\ Since $ 0 < \mu (B(x,r))  < \infty  $ for all $ x \in X $ and all $ r \in \left( 0,\infty\right)$, then the set $X \backslash H$ is dense in $X$. Hence, for $x \in H$, there exists $x_n \in X \backslash H$ converging to $x$.\\
 Then,
	$$
	\left|u\left(x_n\right)-u\left(x_m\right)\right| \leq C [u]_{W^{\alpha, G}_{s}(X,d,\mu)} d\left(x_n, x_m\right)^{\alpha-\frac{s}{p_{0}}} .
	$$
	Therefore, $\left\{u\left(x_n\right)\right\}_{n=1}^{\infty}$ is a Cauchy sequence in $\mathbb{R}$, thus is convergent, let $\widetilde{u}(x)$ denote its limit (which is independent of the choice of the sequence converging to $x$). For $x \in X \backslash H$ we put $\widetilde{u}(x)=u(x)$. Let us fix $x, y \in X$ such that $d(x, y)<\frac{1}{3}$ and let us take $x_n, y_n \in X \backslash H$ such that $x_n \rightarrow x$ and $y_n \rightarrow y$.\\
 Then,
	$$
	u\left(x_n\right) \rightarrow \widetilde{u}(x) \quad \textit{ and } \quad u\left(y_n\right) \rightarrow \widetilde{u}(y) .
	$$
	Hence,
	$$
	\begin{aligned}
	\left|\widetilde{u}(x)-\widetilde{u}(y)\right|& \leq\left|u\left(x_n\right)-\widetilde{u}(x)\right|+\left|u\left(y_n\right)-\widetilde{u}(y)\right|+\left|u\left(x_n\right)-u\left(y_n\right)\right|\\& \leq\left|u\left(x_n\right)-\widetilde{u}(x)\right|+\left|u\left(y_n\right)-\widetilde{u}(y)\right|\\&+C [u]_{W^{\alpha, G}_{s}(X,d,\mu)} d\left(x_n, y_n\right)^{\alpha-\frac{s}{p_{0}}} .
	\end{aligned}
	$$
	Then, by passing to the limit $n \rightarrow \infty$ we get
	$$
	\left|\widetilde{u}(x)-\widetilde{u}(y)\right| \leq C [u]_{W^{\alpha, G}_{s}(X,d,\mu)} d(x, y)^{\alpha-\frac{s}{p_{0}}},
	$$
which completes the proof.
\end{proof}
\begin{theorem}\label{3theo2}
Let $G$ be an Orlicz function such that $G^{-1}(xy)\leq x^{\frac{1}{p_{0}}}G^{-1}(y)$ if $x\leq 1$ and  $G^{-1}(xy)\leq x^{\frac{1}{p^{0}}}G^{-1}(y)$ if $x\geq 1$, and let $s,\alpha>0$ such that $\alpha p_{0}>s$. Then there is a positive constant $\widetilde{C}$ such that for all $u \in W_s^{\alpha, G}(X, d, \mu)$ we have
	\begin{equation}\label{3eqq5}
	\left\|\widetilde{u}\right\|_{C^{0, \alpha-\frac{s}{p_{0}}}(X, d)} \leq \widetilde{C}\|u\|_{W_s^{\alpha, G}(X, d, \mu)},
	\end{equation}
	where $\widetilde{u}$ is a continuous representative of $u$.
\end{theorem}
\begin{proof}
	Let $z \in X$, then there exists $w \in B(z, \frac{1}{3})$ such that 
	
	\begin{equation}\label{3eqq6}
	G\left( \left|\widetilde{u}(w)\right|\right)  \leq\frac{1}{\mu(B(z,\frac{1}{3}))} \int_{B(z,  \frac{1}{3})}G\left(|u(y)|\right) d \mu(y).
	\end{equation}
	If $\|u\|_{L^{G}(X, \mu)}=0$, the inequality \eqref{3eqq5}  is trivial.\newline
	We suppose that $\|u\|_{L^{G}(X, \mu)}\in (0,1]$, then by \eqref{3eqq6} and \eqref{1GG4} we have
	$$
	\begin{aligned}
	\left|\widetilde{u}(w)\right| &\leq G^{-1}\left(\frac{1}{\mu(B(z,\frac{1}{3}))} \int_{B(z,\frac{1}{3})}G\left(|u(y)|\right) d \mu(y)\right)\\ &\leq G^{-1}\left(\frac{1}{\mu(B(z,\frac{1}{3}))} \int_{B(z,\frac{1}{3})}G\left(\frac{|u(y)|}{\|u\|_{L^{G}(X, \mu)}}\right) \|u\|_{L^{G}(X, \mu)}^{p_{0}} d \mu(y)\right)\\
	&\leq G^{-1}\left(b3^{s} \int_{B(z,\frac{1}{3})}G\left(\frac{|u(y)|}{\|u\|_{L^{G}(X, \mu)}}\right) \|u\|_{L^{G}(X, \mu)}^{p_{0}} d \mu(y)\right)
	\\& \leq C_{1} \|u\|_{L^G(X, \mu)}G^{-1}\left( \int_{B(z,\frac{1}{3})}G\left(\frac{|u(y)|}{\|u\|_{L^{G}(X, \mu)}}\right) d \mu(y)\right)  \\ &\leq C_{1} G^{-1}(1) \|u\|_{L^G(X, \mu)},
	\end{aligned}
	$$ 
	where $C_{1}=\max\left\lbrace b^{\frac{1}{p_{0}}} 3^{\frac{s}{p_{0}}},b^{\frac{1}{p^{0}}} 3^{\frac{s}{p^{0}}}\right\rbrace .$\newline
	Now we suppose that $\|u\|_{L^{G}(X, \mu)}\geq 1$, then by \eqref{3eqq6} and \eqref{1GG3} we have
	
	$$
	\begin{aligned}
	\left|\widetilde{u}(w)\right| &\leq G^{-1}\left(\frac{1}{\mu(B(z,\frac{1}{3}))} \int_{B(z,\frac{1}{3})}G\left(|u(y)|\right) d \mu(y)\right)\\ &=G^{-1}\left(\frac{1}{\mu(B(z,\frac{1}{3}))} \int_{B(z,\frac{1}{3})}G\left(\|u\|_{L^{G}(X, \mu)}\frac{|u(y)|}{\|u\|_{L^{G}(X, \mu)}}\right) d \mu(y)\right)\\
	& \leq G^{-1}\left(b 3^{s} \int_{B(z,\frac{1}{3})}G\left(\frac{|u(y)|}{\|u\|_{L^{G}(X, \mu)}}\right) \|u\|_{L^{G}(X, \mu)}^{p^{0}} d \mu(y)\right)
	\\&\leq C_{1}\|u\|_{L^G(X, \mu)}G^{-1}\left( \int_{B(z,\frac{1}{3})}G\left(\frac{|u(y)|}{\|u\|_{L^{G}(X, \mu)}}\right) d \mu(y)\right)  \\ &\leq C_{1} G^{-1}(1) \|u\|_{L^G(X, \mu)}.
	\end{aligned}
	$$
	Then in both cases, we have 
	\begin{equation}\label{3eqqq61}
	 \left|\widetilde{u}(w)\right| \leq C_{2} \|u\|_{L^G(X, \mu)}, 
	\end{equation}
	where $C_{2}=C_{1} G^{-1}(1).$ Then, by Theorem \ref{3theo1} and the inequality \eqref{3eqqq61}, we have  
	$$
	\begin{aligned}
	\left|\widetilde{u}(z)\right| &\leq\left|\widetilde{u}(z)-\widetilde{u}(w)\right|+\left|\widetilde{u}(w)\right|\\ & \leq C[u]_{W_s^{\alpha, G}(X,d,\mu)} d(z, w)^{\alpha-\frac{s}{p_{0}}}+C_{2}\|u\|_{L^G(X, \mu)} \\
	& \leq C_3\|u\|_{W_s^{\alpha, G}(X, d, \mu)},
	\end{aligned}
	$$
	where $C_3=\max \left\{C 3^{\frac{s}{p_{0}}-\alpha}, C_{2}\right\}$. Hence,
	$$
	\left\|\widetilde{u}\right\|_{C(X, d)} \leq C_3\|u\|_{W_s^{\alpha, G}(X, d, \mu)}.
	$$
	By using Theorem \ref{3theo1} we can control the H\"older semi-norm of $\widetilde{u}$ for all $z, w$ such that $d(z, w)<\frac{1}{3}$. Now, let us estimate the whole H\"older's semi-norm.\\ Let $z, w \in X$ such that $d(z, w) \geq \frac{1}{3}$, then
	\begin{equation}\label{3eqq61}
	\begin{aligned}
	\left|\widetilde{u}(z)-\widetilde{u}(w)\right| & \leq 2\left\|\widetilde{u}\right\|_{C(X, d)} \\
	& =2.3^{\alpha-\frac{s}{p_{0}}}\left\|\widetilde{u}\right\|_{C(X, d)} \frac{1}{3^{\alpha-\frac{s}{p_{0}}}} \\
	& \leq 2.3^{\alpha-\frac{s}{p_{0}}} C_3\|u\|_{W_s^{\alpha, G}(X, d, \mu)} d(z, w)^{\alpha-\frac{s}{p_{0}}}.
	\end{aligned}
	\end{equation}
	Then by Theorem \ref{3theo1} and \eqref{3eqq61} we have $$\sup _{x\neq y\in X} \frac{\left|\widetilde{u}(x)-\widetilde{u}(y)\right|}{d(x, y)^{\alpha-\frac{s}{p_{0}}}}\leq C_4\|u\|_{W_s^{\alpha, G}(X, d, \mu)}, $$
	where $C_{4}=\max\left\lbrace C,2.3^{\alpha-\frac{s}{p_{0}}} C_3\right\rbrace .$\newline
	Hence, we have 
	$$\begin{aligned}
	\left\|\widetilde{u}\right\|_{C^{0, \alpha-\frac{s}{p_{0}}}(X, d)}&=\left\|\widetilde{u}\right\|_{C(X, d)}+\sup _{x\neq y\in X} \frac{\left|\widetilde{u}(x)-\widetilde{u}(y)\right|}{d(x, y)^{\alpha-\frac{s}{p_{0}}}}\\
	&\leq C_3\|u\|_{W_s^{\alpha, G}(X, d, \mu)} + C_4\|u\|_{W_s^{\alpha, G}(X, d, \mu)}\\
	&=\widetilde{C}\|u\|_{W_s^{\alpha, G}(X, d, \mu)},
	\end{aligned}
	$$
	where $\widetilde{C}=C_{3}+C_{4}$.\newline
	The proof is complete.
\end{proof}
\begin{theorem} \label{3theo3}
Let $\alpha>0$, $G$ be an Orlicz function such that  $G^{-1}(xy)\leq x^{\frac{1}{p_{0}}}G^{-1}(y)$ if $x\leq 1$ and  $G^{-1}(xy)\leq x^{\frac{1}{p^{0}}}G^{-1}(y)$ if $x\geq 1$ and suppose that $(X, d, \mu)$ is a lower Ahlfors $\theta$-regular space, where $\theta \geq 0$ and $\alpha p_{0}>\theta-s$. Then, the continuous embedding $W_s^{\alpha, G}(X, d, \mu) \hookrightarrow M^{\beta, G}(X, d, \mu)$ holds, where $\beta=\alpha+(s-\theta) / p_{0}$.
\end{theorem}
\begin{proof}
	Let $u \in W_s^{\alpha, G}(X, d, \mu)$ and let us consider the following set 
	$$
	H=\left\{x \in X: \int_{B(x, 1)} G\left( \frac{|u(x)-u(y)|}{d(x, y)^{\alpha}}\right) \frac{ d \mu(y)}{d(x, y)^{s}} =+\infty \text { or }|u(x)|=+\infty\right\}.
	$$
	It is clear that $\mu(H)=0$. Let $\xi, \eta \in X \backslash H$.\newline
	If $d(\xi, \eta)<\frac{1}{2}$, then from the inequalities \eqref{2eq9} and \eqref{1GG4} we have
	
	$$
	\begin{aligned}
	&|u(\xi)-u(\eta)| \\& \leq\left|u(\xi)-m_u(B(\xi, d(\xi, \eta)))\right|+\left|m_u(B(\xi, d(\xi, \eta)))-u(\eta)\right| \\
	& \leq G^{-1}\left(2 \fint_{B(\xi, d(\xi, \eta))}G\left( |u(\xi)-u(y)|\right)  d \mu(y)\right)\\
	&+G^{-1}\left(2 \fint_{B(\xi, d(\xi, \eta))}G\left( |u(\eta)-u(y)|\right)  d \mu(y)\right)\\
	&= G^{-1}\left( \frac{2}{\mu(B(\xi, d(\xi, \eta)))} \int_{B(\xi, d(\xi, \eta))}G\left( |u(\xi)-u(y)|\right)  d \mu(y)\right)\\
	&+G^{-1}\left( \frac{2}{\mu(B(\xi, d(\xi, \eta)))} \int_{B(\xi, d(\xi, \eta))}G\left( |u(\eta)-u(y)|\right)  d \mu(y)\right)
	\\
	& \leq G^{-1}\left(2bd(\xi, \eta)^{-\theta}\int_{B(\xi, d(\xi, \eta))} G\left( \frac{|u(\xi)-u(y)|}{d(\xi, y)^{\alpha}}\right)  d(\xi, \eta)^{s+\alpha p_{0}} \frac{d \mu(y)}{d(\xi, y)^{s}} \right)\\
	&+ G^{-1}\left(2bd(\xi, \eta)^{-\theta}\int_{B(\eta, 2d(\xi, \eta))} G\left(\frac{|u(\eta)-u(y)|}{d(\eta, y)^{\alpha }}\right)(2 d(\xi, \eta))^{s+\alpha p_{0}} \frac{d \mu(y)}{d(\eta, y)^{s}}\right)
	\\
	&\leq  G^{-1}\left(2bd(\xi, \eta)^{-\theta+s+\alpha p_{0}}\int_{B(\xi, d(\xi, \eta))} G\left( \frac{|u(\xi)-u(y)|}{d(\xi, y)^{\alpha}}\right) \frac{d \mu(y)}{d(\xi, y)^{s}} \right)\\
	&+G^{-1}\left(2b.2^{s+\alpha p_{0}}d(\xi, \eta)^{-\theta+s+\alpha p_{0}} \int_{B(\eta, 2 d(\xi, \eta))} G\left(\frac{|u(\eta)-u(y)|}{d(\eta, y)^{\alpha }}\right) \frac{d \mu(y)}{d(\eta, y)^{s}}\right)\\
	&\leq C G^{-1}\left(d(\xi, \eta)^{-\theta+s+\alpha p_{0}}\int_{B(\xi, d(\xi, \eta))} G\left( \frac{|u(\xi)-u(y)|}{d(\xi, y)^{\alpha}}\right) \frac{d \mu(y)}{d(\xi, y)^{s}} \right)\\
	&+ CG^{-1}\left(d(\xi, \eta)^{-\theta+s+\alpha p_{0}} \int_{B(\eta, 2 d(\xi, \eta))} G\left(\frac{|u(\eta)-u(y)|}{d(\eta, y)^{\alpha }}\right) \frac{d \mu(y)}{d(\eta, y)^{s}}\right)\\
	& \leq C d(\xi, \eta)^{\beta}(f(\xi)+f(\eta)),
	\end{aligned}
	$$
	where $$C=\max\left\lbrace \left( b.2^{s +\alpha p_{0} +1}\right)^{\frac{1}{p_{0}}},\left( b.2^{s +\alpha p_{0}+1}\right)^{\frac{1}{p^{0}}}\right\rbrace, $$ $$\beta= \alpha+(s-\theta) / p_{0}$$ and
	$$
	f(x)=G^{-1}\left(\int_{B(x, 1)} G\left(\frac{|u(x)-u(y)|}{d(x, y)^{\alpha}}\right)\frac{d \mu(y)}{d(x, y)^{s}}\right) \in L^G(X, \mu) .
	$$
	If $d(\xi, \eta) \geq \frac{1}{2}$, then we have
	$$
	|u(\xi)-u(\eta)| \leq|u(\xi)|+|u(\eta)| \leq 2^\beta d(\xi, \eta)^\beta\left(|u(\xi)|+|u(\eta)|\right) .
	$$
	Hence, the function defined by $h(x)=\max \left\{C f(x), 2^\beta|u(x)|\right\}$ is a generalized $\beta$-gradient of $u$ and we get that the space $W_s^{\alpha, G}(X, d, \mu)$ is continuously embedded in the space $M^{\beta, G}(X, d, \mu)$. The proof is complete.
\end{proof}

\bibliographystyle{plain}

\end{document}